\documentclass[12pt]{article}
\usepackage{amsmath,amsthm}
\usepackage{amssymb}
\usepackage{color}

\newtheorem{theorem}{Theorem}[section]
\newtheorem{lemma}[theorem]{Lemma}

\newtheorem{example}[theorem]{Example}

\newcommand{\contradiction}{{\hbox{%
    \setbox0=\hbox{$\mkern-3mu\times\mkern-3mu$}%
    \setbox1=\hbox to0pt{\hss$\times$\hss}%
    \copy0\raisebox{0.5\wd0}{\copy1}\raisebox{-0.5\wd0}{\box1}\box0
}}}

\usepackage{colortbl}

\renewcommand{\geq}{\geqslant}
\renewcommand{\leq}{\leqslant}

\def\cref#1{Corollary~$\ref{#1}$}

\input epsf

\begin{document}

\title{Direct constructions for general families of\\
cyclic mutually nearly orthogonal Latin squares}
\date{}
\author
{Fatih Demirkale\thanks{Author was also supported in part
by the  Australian Research Council (grant number DP1092868)}\\
Department of Mathematics, Ko\c{c} University\\
Istanbul 34450, Turkey\\
{\tt fdemirkale@ku.edu.tr}\vspace{3mm} \\
Diane Donovan\thanks{Author was also supported in part  by the
Australian Research Council (grant number DP1092868)}\\
Department of Mathematics\\
The University of Queensland\\
Brisbane, QLD Australia\\
{\tt dmd@maths.uq.edu.au}\vspace{3mm} \\
Abdollah Khodkar\thanks{Corresponding author}\\
Department of Mathematics\\
University of West Georgia\\
Carrollton, GA 30118, USA\\
{\tt akhodkar@westga.edu}\\
}

\maketitle
\begin{abstract}
\noindent Two Latin squares $L=[l(i,j)]$ and $M=[m(i,j)]$, of even
order $n$ with entries $\{0,1,2,\ldots,n-1\}$, are said to be
nearly orthogonal if the superimposition of $L$ on $M$ yields an
$n\times n$ array $A=[(l(i,j),m(i,j))]$ in which each ordered pair
$(x,y)$, $0\leq x,y\leq n-1$ and $x\neq y$, occurs at least once
and the ordered pair $(x,x+n/2)$ occurs exactly twice. In this
paper, we present direct constructions for the existence of
general families of three cyclic mutually orthogonal Latin squares
of orders $48k+14$,
$48k+22$, $48k+38$ and $48k+46$. The techniques employed are based
on the principle of Methods of Differences and so we also
establish infinite classes of ``quasi-difference'' sets for these
orders.\vspace{3mm}

\end{abstract}

\section{Introduction}

A {\em Latin square}, $L=[l(i,j)]$, of order $n$ is an $n\times n$
array in which each row and each column contains each of the
symbols $0,1,\dots, n-1$ precisely once. Given two Latin squares
$L=[l(i,j)]$ and $M=[m(i,j)]$, of order $n$, we define the
superimposition of $L$ on $M$ to be the $n\times n$ array
$A=[(l(i,j),m(i,j))]$, so the cell $(i,j)$ of $A$ contains the
ordered pair $(l(i,j),m(i,j))$. The Latin squares $L$ and $M$ are
said to be {\em orthogonal} if each of the ordered pair $(x,y)$,
$0\leq x,y\leq n-1$, occurs in a cell of $A$. A set of $s$ {\em
mutually orthogonal Latin squares} (MOLS(n)) is a set of $s$ Latin
squares which are pairwise orthogonal.

Orthogonal Latin squares have wide ranging applications and have
consequently  been studied with great interest. However, there are
still many open questions relating to their existences. For
instance,  it is known that there does not exist a pair of
MOLS(6), however it is not known if there exists a set of three MOLS(10),
see \cite{colbourn}. In 2012,  Todorov established that there
exists a set of  four MOLS(14), but it is not known if there exists a set of five
MOLS(14), see \cite{todorov}. The order 22
is the largest order for which it is not known if there exists a
set of four MOLS(22).

In 2002, Raghavarao, Shrikhande and Shrikhande
suggested \cite{RSS} that  given the importance of their applications in experimental
design, the definition of MOLSs could be varied slightly to deal
with orders for which MOLSs are not known to exist. They suggested
that the orthogonality condition could be adapted in such a way
that identical pairs did not occur, $n$ specified pairs occurred
twice and all other pairs occurred precisely once.

Two Latin squares $L=[l(i,j)]$ and $M=[m(i,j)]$, of even order
$n$, are said to be {\em nearly orthogonal} \cite{RSS} if the
superimposition of $L$ on $M$ yields an $n\times n$ array
$A=[(l(i,j),m(i,j))]$ in which each ordered pair $(x,y)$, $0\leq
x,y\leq n-1$ and $x\neq y$, occurs at least once and the ordered
pair $(x,x+n/2)$ occurs exactly twice. As a consequence of the
definition, we note that none of the $n$ ordered pairs $(x,x)$,
$0\leq x\leq n-1$,  occurs in $A$. A set of $s$ {\em  mutually
nearly orthogonal Latin squares} (MNOLS(n)) is a set of $s$ Latin
squares which are pairwise nearly orthogonal.

It is known that there exist a set of three MNOLS(6),
a set of four MNOLS(10) and a set of four MNOLS(14), but no set of four MNOLS(6), see \cite{LvR, RSS},
raising interesting questions about the existence of sets of MNOLS($n$).

Raghavarao, Shrikhande and Shrikhande, established the following
upper bound on the size of a set of MNOLSs of order $n$.

\begin{theorem} \cite{RSS} Let $L_1,L_2,\dots, L_t$ be $t$ Latin
squares of order $n=2m$ on symbols $\{0,1,2,\dots,n-1\}$ such that
each pair of Latin squares is nearly orthogonal. Then
\begin{eqnarray*}
t\leq \left\{\begin{array}{ll} \frac{n}{2}+1, &\mbox{if }n\equiv
2(\mbox{mod }4),\mbox{ or},\\
\frac{n}{2}, &\mbox{if }n\equiv 0(\mbox{mod }4).\end{array}\right.
\end{eqnarray*}
\end{theorem}

In the  paper \cite{RSS}, Raghavarao, Shrikhande and Shrikhande
used the principle of the Method of Differences to established a
construction for MNOLSs:

\begin{theorem}\label{rss} \cite{RSS} Let there exist $t$ column
vectors of length $2m$, denoted ${\cal C}_s=[c_s(i,0)]$, for $0\leq
i\leq 2m-1$ and $s=1,2,\dots,t$, where each column vector is a
permutation of the elements of the cyclic group ${\mathbb Z}_{2m}$.
Furthermore, suppose for every $ s\neq s'$, $1\leq s,s'\leq t$,
among the $2m$ differences
$c_s(1,0)-c_{s'}(1,0),c_s(2,0)-c_{s'}(2,0),\dots,c_s(2m-1,0)-c_{s'}(2m-1,0)$ modulo
$2m$, $m$ occurs twice and all other non-zero elements of ${\mathbb
Z}_{2m}$ occur once. Then ${\cal L}_s=[l_s(i,j)]$, where
$l_s(i,j)\equiv(c_s(i,0)+j)(\mbox{mod }2m)$ for $0\leq j\leq 2m-1$
and $s=1,2,\dots,t$, forms a set of $t$ MNOLS(2m).
\end{theorem}

The MNOLSs, ${\cal L}_s$, constructed as in Theorem \ref{rss} will be
termed {\em cyclic} MNOLSs.
In \cite{LvR}, it was proven that there exist two cyclic MNOLSs of
order $2m$ for all $m\geq 2$.  In the same paper, it was also proven
that there exist three MNOLS(2m) for all $2m\geq 358$.
But the existence of three cyclic MNOLSs of order $2m$ is still
open.

In this paper, we prove the existence of general families of
column vectors which establish the existence of  three
cyclic MNOLSs of orders $48k+14$, $48k+22$, $48k+38$ and $48k+46$
for all $k\in {\mathbb Z}^+\cup\{0\}$.
Since the constructions are based on
the principle of Methods of Differences the paper also establishes
infinite classes of ``quasi-difference'' sets for these orders,
which may have  applications in the theory of orthomorphisms, see \cite{wanless,evans}.

The Latin squares generated here will be of even order and cyclic.
In addition, they will all have the following property. We will
say that the column vector ${\cal C}$ has the {\em reflection
property}, if $c(i,0)+c(n-1-i,0)\equiv n-1$ (mod $n$) for all
$i=0,\dots,(n-2)/2.$ Further we will say that MNOLSs developed
from such column vector, also have the reflection property.

\begin{example}\label{example1}
Let $V_1=\{(i,0,i),\mid 0\leq i\leq (n-2)/2\}$ and
$\overline{V_1}=\{(n-1-i,0,n-1-i))\mid 0\leq i\leq (n-2)/2\}.$
Then ${\cal C}_1=V_1\cup \overline{V_1}$ has the reflection
property. Let ${\cal L}_1=[l_1(i,j)]$, where
$l_1(i,j)\equiv(c_1(i,0)+j)(\mbox{mod }n)$
for $0\leq j\leq n-1$. Then ${\cal L}_1$ is also said to have the reflection property.
\end{example}

In subsequent sections, the symbol $\contradiction$ has been used to represent ``a contradiction''.

\section{Three cyclic MNOLSs of Order $48k+14$, $k\geq 0$}

In this section we construct two cyclic Latin squares ${\cal L}_2$ and
${\cal L}_3$ both of order $48k+14$ and show that ${\cal L}_1$
(constructed by Example \ref{example1} with $n=48k+14$),
${\cal L}_2$ and ${\cal L}_3$ are cyclic MNOLSs.
The following lemma is crucial in this section.

\begin{lemma}\label{lemma:gcd.1}
Let $k$ be an integer. Working modulo $48k+14$,
{\bf 1.} $\gcd(6k+2,24k+7)=1$;
{\bf 2.} $\gcd(12k+5,48k+14)=1$;
{\bf 3.} $\gcd(6k+1,24k+7)=1$;
{\bf 4.} $\gcd(12k+3,48k+14)=1$.
\end{lemma}

\begin{proof}
The following equations verify the statements given in the lemma:
$
{\bf 1.\ } 4(6k+2)-(24k+7)=1;\
{\bf 2.\ } (8k+3)(12k+5)-(2k+1)(48k+14)=1;\
{\bf 3.\ } (4k+1)(24k+7)-(16k+6)(6k+1)=1;\
{\bf 4.\ } (24k+5)(12k+3)-(6k+1)(48k+14)=1.$
\end{proof}

\noindent Working modulo $48k+14$ we define the $(24k+7)\times 1$ matrices (column vectors)
$V_\alpha=[v_\alpha(i,0)]$, $\alpha=2,3$, by
\begin{eqnarray}
V_2&=&\{(2i,0,6k+1+i(12k+4))\mid 0\leq i\leq 12k+3\}\cup\nonumber\\
&&\{(2i+1,0,12k+3+i(12k+4))\mid 0\leq i\leq 12k+2\},\label{c2.1}\\
V_3&=&\{(2i,0,6k+2+i(12k+5))\mid 0\leq i\leq 12k+3\}\cup\nonumber\\
&&\{(2i+1,0,24k+8+i(12k+5))\mid 0\leq i\leq 12k+2\}.\label{c3.1}
\end{eqnarray}
 For $\alpha=2,3$, let
${\cal C}_\alpha=V_\alpha\cup \overline{V_\alpha}$, where
$$\overline{V_\alpha}=\{(48k+13-i,0,48k+13-v_\alpha(i,0))\mid 0\leq
i\leq 24k+6\}.$$
Note that ${\cal C}_\alpha$ has the reflection property.
Now define ${\cal L}_\alpha=[l_\alpha(i,j)]$, where
$l_\alpha(i,j)\equiv {\cal C}_\alpha(i,0)+j$ (mod $48k+14$) for $0\leq i,j\leq 48k+13$.

\begin{lemma}\label{distinct.entries.L2.1}
The array ${\cal L}_2$ is a Latin square of order $48k+14$, $k\geq 0$.
\end{lemma}

\begin{proof}
 The entries in $V_2$ are all distinct   as verified by Equation \ref{eq:14-1}  for the case
 rows $2i$ and $2j$, where $0\leq i,j\leq 12k+3$, Equation \ref{eq:14-2}  for the case
rows $2i+1$ and $2j+1$,
where $0\leq i,j\leq 12k+2$, and  Equation \ref{eq:14-3}  for the case
rows $2i$ and $2j+1$, where $0\leq i\leq 12k+3$ and $0\leq j\leq 12k+2$.
\begin{eqnarray}
6k+1+i(12k+4)&\equiv& 6k+1+j(12k+4)\;(\mbox{mod } 48k+14),\nonumber\\
\Rightarrow(j-i)(6k+2)&\equiv& 0\;(\mbox{mod } 24k+7),\label{eq:14-1}\ \contradiction;\\
12k+3+i(12k+4)&\equiv& 12k+3+j(12k+4)\;(\mbox{mod } 48k+14),\nonumber\\
\Rightarrow(j-i)(6k+2)&\equiv& 0\;(\mbox{mod } 24k+7),\label{eq:14-2}\ \contradiction;\\
6k+1+i(12k+4)&\equiv& 12k+3+j(12k+4)\;(\mbox{mod } 48k+14),\nonumber\\
\Rightarrow(j-i)(6k+2)&\equiv& -3k-1\;(\mbox{mod } 24k+7),\nonumber\\
\Rightarrow j-i&\equiv& 4(-3k-1)\;(\mbox{mod } 24k+7),\label{eq:14-3}\end{eqnarray}
implying $j-i=12k+3$, or
$j=12k+3+i>12k+2$, which leads to a contradiction.

For any two rows containing entries $x$ and $y$ in $V_2$, parity conditions  and the following equations can be used to verify  $x+y+1\not\equiv 0$ (mod $48k+14$), specifically
 Equation \ref{eq:14-4}
for rows $2i$ and $2j$, Equation \ref{eq:14-5} for rows $2i+1$ and $2j+1$ and Equation \ref{eq:14-6} for rows $2i$ and $2j+1$.
\begin{eqnarray}
2(6k+1)+(i+j)(12k+4)+1&\equiv& 0\;(\mbox{mod } 48k+14),\nonumber\\
\Rightarrow 12k+3+(i+j)(12k+4)&\equiv& 0\;(\mbox{mod } 48k+14),\label{eq:14-4}\ \contradiction;\\
2(12k+3)+(i+j)(12k+4)+1&\equiv& 0\;(\mbox{mod } 48k+14),\nonumber\\
\Rightarrow 24k+7+(i+j)(12k+4)&\equiv& 0\;(\mbox{mod } 48k+14),\label{eq:14-5}\ \contradiction;\\
6k+1+12k+3+(i+j)(12k+4)+1&\equiv& 0\;(\mbox{mod } 48k+14),\nonumber\\
\Rightarrow 18k+5+(i+j)(12k+4)&\equiv& 0\;(\mbox{mod } 48k+14),\label{eq:14-6}\ \contradiction.
\end{eqnarray}
Thus the entries of ${\cal C}_2$ are all distinct and so
${\cal L}_2$ is a Latin square of order $48k+14$.
\end{proof}

\begin{lemma}\label{distinct.entries.L3.1}
The array ${\cal L}_3$ is a Latin square of order $48k+14$, $k\geq 0$.
\end{lemma}

\begin{proof}

The entries in $V_3$ are all distinct  as verified by Equation \ref{eq:14-7}  for the case
 rows $2i$ and $2j$, where $0\leq i,j\leq 12k+3$, Equation \ref{eq:14-8}  for the case
rows $2i+1$ and $2j+1$,
where $0\leq i,j\leq 12k+2$, and  Equation \ref{eq:14-9}  for the case
rows $2i$ and $2j+1$, where $0\leq i\leq 12k+3$ and $0\leq j\leq 12k+2$.
\begin{eqnarray}
6k+2+i(12k+5)&\equiv& 6k+2+j(12k+5)\;(\mbox{mod } 48k+14),\nonumber\\
\Rightarrow (j-i)(12k+5)&\equiv& 0\;(\mbox{mod } 48k+14),\label{eq:14-7}\ \contradiction;\\
24k+8+i(12k+5)&\equiv& 24k+8+j(12k+5)\;(\mbox{mod } 48k+14),\nonumber\\
\Rightarrow (j-i)(12k+5)&\equiv& 0\;(\mbox{mod } 48k+14), \label{eq:14-8}\ \contradiction;\\
6k+2+i(12k+5)&\equiv& 24k+8+j(12k+5)\;(\mbox{mod } 48k+14),\nonumber\\
j-i&\equiv& (-18k-6)(8k+3)\equiv 36k+10 \;(\mbox{mod } 48k+14), \label{eq:14-9}\end{eqnarray}
implying $j=36k+10+i>12k+2$ or $j=-12k-4+i<0$, which leads to a contradiction.

For any two rows containing entries $x$ and $y$ in $V_3$, $x+y+1\not\equiv 0$ (mod $48k+14$) as verified by Equation \ref{eq:14-10}
for rows $2i$ and $2j$, Equation \ref{eq:14-11} for rows $2i+1$ and $2j+1$ and Equation \ref{eq:14-12} for rows $2i$ and $2j+1$.
\begin{eqnarray}
2(6k+2)+(i+j)(12k+5)+1&\equiv& 0\;(\mbox{mod } 48k+14),\nonumber\\
\Rightarrow (i+j+1)(12k+5)&\equiv& 0\;(\mbox{mod } 48k+14),\nonumber\\
\Rightarrow i+j+1&\equiv& 0\;(\mbox{mod } 48k+14),
\label{eq:14-10}\ \contradiction;\\
2(24k+8)+(i+j)(12k+5)+1&\equiv& 0\;(\mbox{mod } 48k+14),\nonumber\\
\Rightarrow (i+j)(12k+5)&\equiv& -48k-17\;(\mbox{mod } 48k+14),\nonumber\\
\Rightarrow i+j\equiv (-48k-17)(8k+3)&\equiv& 24k+5\;(\mbox{mod } 48k+14)\label{eq:14-11}\ \contradiction;\\
6k+2+24k+8+(i+j)(12k+5)+1&\equiv& 0\;(\mbox{mod } 48k+14),\nonumber\\
\Rightarrow+(i+j)(12k+5)&\equiv& -30k-11\;(\mbox{mod } 48k+14),\nonumber\\
\Rightarrow  i+j\equiv (-30k-11)(8k+3)&\equiv& 36k+9 \;(\mbox{mod } 48k+14),
\label{eq:14-12}\ \contradiction.
\end{eqnarray}
Thus the entries of ${\cal C}_3$ are all distinct and so
${\cal L}_3$ is a Latin square of order $48k+14$.
\end{proof}

\begin{theorem}\label{theorem.1}
The Latin squares ${\cal L}_1$, ${\cal L}_2$ and ${\cal L}_3$ are
cyclic MNOLSs of order $48k+14$, $k\geq 0$.
\end{theorem}

\begin{proof}

Respectively, the differences between entries in rows $2i$ and $2i+1$ of $V_2$ and $V_1$, are
\begin{eqnarray*}
6k+1+i(12k+4)-2i&\equiv&(2i+1)(6k+1)\; (\mbox{mod } 48k+14), \\
12k+3+i(12k+4)-2i-1&\equiv& (2i+2)(6k+1)\; (\mbox{mod } 48k+14).
\end{eqnarray*}
These differences are all non-zero since in the first instance
 $(2i+1)(6k+1)$ is
odd and $48k+14$ is even and in the second instance if the difference $(2i+2)(6k+1)\equiv 0$
(mod $48k+14$), then by Lemma \ref{lemma:gcd.1}, $i+1\equiv 0$ (mod $24k+7$),
which implies $i=24k+6$, a contradiction.

\noindent The differences are all distinct  as verified by Equation \ref{eq:14-13} for
 rows $2i$ and $2j$ and for rows $2i+1$ and $2j+1$, and using a parity argument in Equation \ref{eq:14-14}  for rows $2i$ and $2j+1$.
 \begin{eqnarray}
 2(j-i)(6k+1)&\equiv& 0\;(\mbox{mod } 48k+14),\nonumber\\
 \Rightarrow (j-i)(6k+1)&\equiv& 0\;(\mbox{mod } 24k+7),
\label{eq:14-13}\ \contradiction;\\
2(j-i)(6k+1)&\equiv &-6k-1\;(\mbox{mod } 48k+14),\label{eq:14-14}\ \contradiction.
 \end{eqnarray}

\noindent In addition, any two distinct differences $x$ and
$y$, produced by corresponding rows of $V_2$ and $V_1$, satisfy
 $x+y\not\equiv 0$ (mod $48k+14$), as verified by Equation \ref{eq:14-15} for rows
$2i$ and $2j$, Equation \ref{eq:14-16} for rows $2i+1$ and $2j+1$ and parity arguments
together with Equation \ref{eq:14-17} for rows $2i$ and $2j+1$. In all such cases $x+y$ is congruent to
\begin{eqnarray}
(i+j+1)(6k+1)&\equiv& 0\;(\mbox{mod } 24k+7),\nonumber\\
\Rightarrow i+j+1 &\equiv& 0\;(\mbox{mod } 24k+7),
\label{eq:14-15}\ \contradiction;\\
(i+j+2)(6k+1)&\equiv& 0\;(\mbox{mod } 24k+7),\nonumber\\
\Rightarrow i+j+2 &\equiv& 0\;(\mbox{mod } 24k+7),
\label{eq:14-16}\ \contradiction;\\
6k+1+2(i+j+1)(6k+1)&\equiv& 0\;(\mbox{mod } 48k+14),\label{eq:14-17}\ \contradiction.
\end{eqnarray}

Respectively, the differences between entries in rows $2i$ and $2i+1$ of $V_3$ and $V_1$, are
\begin{eqnarray*}
6k+2+i(12k+5)-2i&\equiv& 6k+2+i(12k+3)\; (\mbox{mod } 48k+14),\\
24k+8+i(12k+5)-2i-1&\equiv& 24k+7+i(12k+3)\; (\mbox{mod } 48k+14).
\end{eqnarray*}

\noindent Equations \ref{eq:14-18} and \ref{eq:14-19} verify that these differences are all non-zero.
\begin{eqnarray}
 i\equiv (-6k-2)(24k+5)&\equiv& 12k+4\;(\mbox{mod } 48k+14),
\label{eq:14-18}\ \contradiction\\
 i\equiv (-24k-7)(24k+5)&\equiv& 24k+7\;(\mbox{mod } 48k+14),
\label{eq:14-19}\ \contradiction.
\end{eqnarray}
If two differences produced by rows $2i$ and $2j$ or by rows $2i+1$ and $2j+1$ are equal, then
$(j-i)(12k+3)\equiv 0$ (mod $48k+14$).
Now by Lemma \ref{lemma:gcd.1}, $i=j$. Equation \ref{eq:14-20} verifies that   two differences produced by rows $2i$ and $2j+1$ are never equal.
\begin{eqnarray}
(j-i)(12k+3)&\equiv& -18k-5\;(\mbox{mod } 48k+14),\nonumber\\
\Rightarrow j-i\equiv (-18k-5)(24k+5)&\equiv& 12k+3\; (\mbox{mod } 48k+14), \label{eq:14-20}
\end{eqnarray}
implying $j=i+12k+3>12k+2$, which leads to a contradiction.

\noindent In addition, any two distinct differences $x$ and
$y$, produced by corresponding rows of $V_3$ and $V_1$, satisfy
 $x+y\not\equiv 0$ (mod $48k+14$), as verified by Equation \ref{eq:14-21} for rows
$2i$ and $2j$, Equation \ref{eq:14-22} for rows $2i+1$ and $2j+1$ and
parity arguments together with Equation \ref{eq:14-23} for rows $2i$ and $2j+1$. In all such cases $x+y$ is congruent to
\begin{eqnarray}
12k+4+(i+j)(12k+3)&\equiv& 0\;(\mbox{mod } 48k+14),\nonumber\\
\Rightarrow i+j\equiv (-12k-4)(24k+5)&\equiv &24k+8\;(\mbox{mod } 48k+14),\label{eq:14-21}\ \contradiction;\\
(i+j)(12k+3)&\equiv& 0\;(\mbox{mod } 48k+14),\label{eq:14-22}\ \contradiction;\\
30k+9+(i+j)(12k+3)&\equiv& 0\;(\mbox{mod } 48k+14),\nonumber\\
\Rightarrow i+j\equiv(-30k-9)(24k+5)&\equiv& 36k+11\;(\mbox{mod } 48k+14),\label{eq:14-23}\ \contradiction.
\end{eqnarray}

Respectively, the differences between entries are in rows $2i$ and $2i+1$ of $V_3$ and $V_2$, are
\begin{eqnarray*}
6k+2-6k-1+i(12k+5-12k-4)&\equiv& i+1\; (\mbox{mod } 48k+14),\\
24k+8-12k-3+i(12k+5-12k-4)&\equiv& 12k+5+i\; (\mbox{mod } 48k+14).
\end{eqnarray*}

\noindent These are all non-zero and distinct.   In addition,
any two distinct differences $x$ and $y$ satisfy $x+y\not\equiv 0$ (mod $48k+14$).

By Lemmas \ref{distinct.entries.L2.1}, \ref{distinct.entries.L3.1} and the above arguments, the
Latin squares ${\cal L}_1$, ${\cal L}_2$ and ${\cal L}_3$ are cyclic
MNOLSs.
\end{proof}

\section{Three cyclic MNOLSs of Order  $48k+22$, $k\geq 0$}

In this section we construct two cyclic Latin squares ${\cal L}_2$ and
${\cal L}_3$ both of order $48k+22$ and show that ${\cal L}_1$
(constructed by Example \ref{example1} with $n=48k+22$),
${\cal L}_2$ and ${\cal L}_3$ are cyclic MNOLSs.
The following lemma is crucial in this section.

\begin{lemma}\label{lemma:gcd.2}
Let $k$ be an integer. Working modulo $48k+22$,
{\bf 1.} $\gcd(6k+3,24k+11)=1$;
{\bf 2.} $gcd(12k+7,48k+22)=1$;
{\bf 3.} $\gcd(6k+2,24k+11)=1$;
{\bf 4.} $\gcd(12k+5,48k+22)=1$.
\end{lemma}

\begin{proof}
The following equations verify the statements given in the lemma:
${\bf 1.\ } 4(6k+3)-(24k+11) = 1;\
{\bf 2.\ } (2k+1)(48k+22)-(8k+3)(12k+7) = 1;\
{\bf 3.\ } (2k+1)(24k+11)-(8k+5)(6k+2) = 1;\
{\bf 4.\ } (24k+9)(12k+5)-(6k+2)(48k+22) = 1.$
\end{proof}

\noindent Working modulo $48k+22$ we define the $(24k+11)\times 1$ matrices (column vectors)
$V_\alpha=[v_\alpha(i,0)]$, $\alpha=2,3$, by
\begin{eqnarray}
V_2&=&\{(2i,0,30k+13+i(12k+6))\mid 0\leq i\leq 12k+5\}\cup\nonumber\\
&&\{(2i+1,0,12k+5+i(12k+6))\mid 0\leq i\leq 12k+4\},\label{c2.2}\\
V_3&=&\{(2i,0,30k+14+i(12k+7))\mid 0\leq i\leq 12k+5\}\cup\nonumber\\
&&\{(2i+1,0,24k+12+i(12k+7))\mid 0\leq i\leq 12k+4\}.\label{c3.2}
\end{eqnarray}
 For $\alpha=2,3$, let
${\cal C}_\alpha=V_\alpha\cup \overline{V_\alpha}$, where
$$\overline{V_\alpha}=\{(48k+21-i,0,48k+21-v_\alpha(i,0))\mid 0\leq
i\leq 24k+10\}.$$
Note that ${\cal C}_\alpha$ has the reflection property.
Now define ${\cal L}_\alpha=[l_\alpha(i,j)]$, where
$l_\alpha(i,j)\equiv {\cal C}_\alpha(i,0)+j$ (mod $48k+22$) for $0\leq i,j\leq 48k+21$.

\begin{lemma}\label{distinct.entries.L2.2}
The array ${\cal L}_2$ is a Latin square of order $48k+22$, for $k\geq 0$.
\end{lemma}

\begin{proof}
 The entries in $V_2$ are all distinct   as verified by Equation \ref{eq:22-1}  for the case
 rows $2i$ and $2j$, where $0\leq i,j\leq 12k+5$, Equation \ref{eq:22-2}  for the case
rows $2i+1$ and $2j+1$,
where $0\leq i,j\leq 12k+4$, and  Equation \ref{eq:22-3}  for the case
rows $2i$ and $2j+1$, where $0\leq i\leq 12k+5$ and $0\leq j\leq 12k+4$.
\begin{eqnarray}
 30k+13+i(12k+6)&\equiv&  30k+13+j(12k+6)\;(\mbox{mod } 48k+22),\nonumber\\
\Rightarrow (j-i)(6k+3)&\equiv&  0\;(\mbox{mod }24k+11), \ \contradiction;\label{eq:22-1}\\
 12k+5+i(12k+6)&\equiv& 12k+5+j(12k+6)\;(\mbox{mod } 48k+22)\nonumber\\
\Rightarrow  (j-i)(6k+3)&\equiv&  0\;(\mbox{mod }24k+11), \ \contradiction;\label{eq:22-2}\\
 30k+13+i(12k+6)&\equiv&  12k+5+j(12k+6)\;(\mbox{mod } 48k+22),\nonumber\\
\Rightarrow  (j-i)(6k+3)&\equiv&  9k+4\;(\mbox{mod } 24k+11), \nonumber\\
\Rightarrow   j-i&\equiv& 4( 9k+4)\;(\mbox{mod } 24k+11), \ \contradiction.\label{eq:22-3}
\end{eqnarray}

For any two rows containing entries $x$ and $y$ in $V_2$, parity conditions and the following equations can be used to verify that $x+y+1\not\equiv 0$ (mod $48k+22$), specifically  Equation \ref{eq:22-4}
for rows $2i$ and $2j$, Equation \ref{eq:22-5} for rows  $2i+1$ and $2j+1$ and Equation \ref{eq:22-6} for rows $2i$ and $2j+1$.
\begin{eqnarray}
 2(30k+13)+(i+j)(12k+6)+1&\equiv& 0\;(\mbox{mod } 48k+22),\nonumber\\
\Rightarrow   12k+5+(i+j)(12k+6)&\equiv &0\; (\mbox{mod }48k+22),\ \contradiction;\label{eq:22-4}\\
 2(12k+5)+(i+j)(12k+6)+1&\equiv& 0\;(\mbox{mod } 48k+22),\nonumber\\
\Rightarrow   24k+11+(i+j)(12k+6)&\equiv& 0\;(\mbox{mod } 48k+22),\ \contradiction;\label{eq:22-5}\\
 30k+13+12k+5+(i+j)(12k+6)+1&\equiv& 0\;(\mbox{mod } 48k+22),\nonumber\\
\Rightarrow   42k+19+(i+j)(12k+6)&\equiv& 0\;(\mbox{mod } 48k+22),\ \contradiction.\label{eq:22-6}
\end{eqnarray}
Thus the entries of ${\cal C}_2$ are all distinct and so
${\cal L}_2$ is a Latin square of order $48k+22$.
\end{proof}

\begin{lemma}\label{distinct.entries.L3.2}
The array ${\cal L}_3$ is a Latin square of order $48k+22$, for $k\geq 0$.
\end{lemma}

\begin{proof}
 The entries in $V_3$ are all distinct  as verified by Equation \ref{eq:22-7}  for the case
 rows $2i$ and $2j$, where $0\leq i,j\leq 12k+5$, Equation \ref{eq:22-8}  for the case
rows $2i+1$ and $2j+1$,
where $0\leq i,j\leq 12k+4$, and  Equation \ref{eq:22-9}  for the case
rows $2i$ and $2j+1$, where $0\leq i\leq 12k+5$ and $0\leq j\leq 12k+4$.
\begin{eqnarray}
 30k+14+i(12k+7)&\equiv& 30k+14+j(12k+7)\;(\mbox{mod } 48k+22),\nonumber\\
\Rightarrow  (j-i)(12k+7)&\equiv &0\;(\mbox{mod } 48k+22),\ \contradiction;\label{eq:22-7} \\
 24k+12+i(12k+7)&\equiv& 24k+12+j(12k+7)\;(\mbox{mod } 48k+22),\nonumber\\
\Rightarrow (j-i)(12k+7)&\equiv &0\;(\mbox{mod } 48k+22),\ \contradiction;\label{eq:22-8}\\
 30k+14+i(12k+7)&\equiv& 24k+12+j(12k+7)\;(\mbox{mod } 48k+22),\nonumber\\
\Rightarrow  (j-i)(12k+7)&\equiv &6k+2\;(\mbox{mod } 48k+22),\nonumber\\
\Rightarrow  j-i\equiv (-8k-3)(6k+2)&\equiv& 36k+16\;(\mbox{mod } 48k+22),\label{eq:22-9}
\end{eqnarray}
implying $j=36k+16+i>12k+4$ or $j=-12k-6+i<0$, a contradiction.

For any two rows  containing entries $x$ and $y$ in $V_3$, $x+y+1\not\equiv 0$ (mod $48k+22$) as verified by Equation \ref{eq:22-10}
for rows $2i$ and $2j$, Equation \ref{eq:22-11} for rows $2i+1$ and $2j+1$ and Equation \ref{eq:22-12} for rows $2i$ and $2j+1$.
\begin{eqnarray}
 2(30k+14)+(i+j)(12k+7)+1&\equiv& 0\;(\mbox{mod } 48k+22),\nonumber\\
\Rightarrow   (i+j+1)(12k+7)&\equiv& 0\;(\mbox{mod } 48k+22),\ \contradiction;\label{eq:22-10}\\
 2(24k+12)+(i+j)(12k+7)+1&\equiv& 0\;(\mbox{mod } 48k+22),\nonumber\\
\Rightarrow  (i+j)(12k+7)&\equiv& -3\;(\mbox{mod } 48k+22),\nonumber\\
\Rightarrow  i+j\equiv (-3)(-8k-3)&\equiv& 24k+9\;(\mbox{mod } 48k+22),
\ \contradiction;\label{eq:22-11}\\
30k+14+24k+12+(i+j)(12k+7)+1&\equiv& 0\;(\mbox{mod } 48k+22),\nonumber\\
\Rightarrow  (i+j)(12k+7))&\equiv& -6k-5\;(\mbox{mod } 48k+22),\nonumber\\
\Rightarrow  i+j\equiv (-6k-5)(-8k-3)&\equiv& 36k+15\;(\mbox{mod } 48k+22),
\ \contradiction.\label{eq:22-12}
\end{eqnarray}
Thus the entries of ${\cal C}_3$ are all distinct and so
${\cal L}_3$ is a Latin square of order $48k+22$.
\end{proof}

\begin{theorem}\label{theorem.2}
The Latin squares ${\cal L}_1$, ${\cal L}_2$ and ${\cal L}_3$ are
cyclic MNOLSs of order $48k+22$, $k\geq 0$.
\end{theorem}

\begin{proof} Respectively, for rows $2i$ and $2i+1$ the differences between entries  of $V_2$ and $V_1$, are
\begin{eqnarray*}
\begin{array}{rcl}
30k+13+i(12k+6)-2i&\equiv& 30k+13+i(12k+4)\; (\mbox{mod } 48k+22),\\
12k+5+i(12k+6)-2i-1&\equiv& 12k+4+i(12k+4)\; (\mbox{mod } 48k+22).
\end{array}
\end{eqnarray*}

\noindent These differences are all non-zero because  in the first instance
since $30k+13$ is odd but $12k+4$ and $48k+22$ are even and in the second
if  $(i+1)(12k+4)\equiv 0$
(mod $48k+22$), then  $i+1\equiv 0$ (mod $24k+11$),
implying $i=24k+10$, a contradiction.

\noindent The differences are all distinct  as verified by Equation \ref{eq:22-13} for
 rows $2i$ and $2j$ and for rows $2i+1$ and $2j+1$, and using a parity argument in Equation \ref{eq:22-14}  for rows $2i$ and $2j+1$.
 \begin{eqnarray}
  (j-i)(12k+4)&\equiv& 0\; (\mbox{mod }48k+22),\nonumber\\
\Rightarrow  (j-i)(6k+2)&\equiv& 0\; (\mbox{mod }24k+11),\ \contradiction;\label{eq:22-13}\\
  (j-i)(12k+4)&\equiv& 18k+9\; (\mbox{mod }48k+22),\ \contradiction.\label{eq:22-14}
 \end{eqnarray}

\noindent In addition,
any two distinct differences $x$ and $y$, produced by corresponding rows in $V_2$ and $V_1$,
satisfy $x+y\not\equiv 0$ (mod $48k+22$) as verified by Equation \ref{eq:22-15} for rows
$2i$ and $2j$, Equation \ref{eq:22-16} for rows $2i+1$ and $2j+1$ and parity condidions in Equation \ref{eq:22-17} for rows $2i$ and $2j+1$. In all such cases $x+y$ equals
\begin{eqnarray}
 (i+j+1)(6k+2)&\equiv& 0\; (\mbox{mod }24k+11),\nonumber\\
\Rightarrow    i+j+1&\equiv& 0\; (\mbox{mod }24k+11),\ \contradiction;
\label{eq:22-15}\\
 (i+j+2)(6k+2)&\equiv& 0\; (\mbox{mod }24k+11),\nonumber\\
\Rightarrow   i+j+2&\equiv& 0\; (\mbox{mod }24k+11),\ \contradiction;
\label{eq:22-16}\\
 (42k+17)+(i+j)(12k+4)&\equiv& 0\; (\mbox{mod }48k+22), \ \contradiction.
\label{eq:22-17}
 \end{eqnarray}

Respectively, the differences between entries in rows $2i$ and $2i+1$ of $V_3$ and $V_1$, are
\begin{eqnarray*}
30k+14+i(12k+7)-2i&\equiv& 30k+14+i(12k+5)\; (\mbox{mod } 48k+22),\\
24k+12+i(12k+7)-2i-1&\equiv& 24k+11+i(12k+5)\; (\mbox{mod } 48k+22).
\end{eqnarray*}
Equations \ref{eq:22-18} and \ref{eq:22-19} verify that these differences are all non-zero.
\begin{eqnarray}
 i\equiv (-30k-14)(24k+9)&\equiv&  12k+6\;(\mbox{mod }48k+22),\ \contradiction;
\label{eq:22-18}\\
 i\equiv (-24k-11)(24k+9)&\equiv& 24k+11\;(\mbox{mod }48k+22),\ \contradiction.
\label{eq:22-19}
\end{eqnarray}

\noindent The differences are all distinct  as verified by Equation \ref{eq:22-20} for
 rows $2i$ and $2j$, and rows $2i+1$ and $2j+1$, and using a parity argument in Equation \ref{eq:22-21}  for rows $2i$ and $2j+1$.
\begin{eqnarray}
(j-i)(12k+5)&\equiv& 0\;(\mbox{mod }48k+22),\ \contradiction;\label{eq:22-20}\\
 (j-i)(12k+5)&\equiv& 6k+3\;(\mbox{mod }48k+22),\nonumber\\
\Rightarrow   j-i\equiv(6k+3)(24k+9)&\equiv&12k+5\;(\mbox{mod }48k+22),\label{eq:22-21}
\end{eqnarray}
a contradiction, since $j=i+12k+5>12k+4.$

\noindent In addition,
any two distinct differences $x$ and $y$,  produced by corresponding rows in $V_3$ and $V_1$, satisfy $x+y\not\equiv 0$ (mod $48k+22$)
as verified by Equation \ref{eq:22-22} for rows
$2i$ and $2j$, Equation \ref{eq:22-23} for rows $2i+1$ and $2j+1$ and Equation \ref{eq:22-24} for rows $2i$ and $2j+1$. In all such cases $x+y$ equals
\begin{eqnarray}
  12k+6+(i+j)(12k+5)&\equiv& 0\;(\mbox{mod }48k+22),\nonumber\\
\Rightarrow  i+j\equiv(-12k-6)(24k+9)&\equiv& 24k+12\;(\mbox{mod }48k+22),\ \contradiction;\label{eq:22-22}\\
 (i+j)(12k+5)&\equiv& 0\;(\mbox{mod }48k+22),\ \contradiction;\label{eq:22-23}\\
 (6k+3)+(i+j)(12k+5)&\equiv& 0\;(\mbox{mod }48k+22)\nonumber\\
\Rightarrow  i+j\equiv(-6k-3)(24k+9)&\equiv& 36k+17\;(\mbox{mod }48k+22),\ \contradiction.\label{eq:22-24}
\end{eqnarray}

Respectively, the differences between entries are in rows $2i$ and $2i+1$ of $V_3$ and $V_2$, are
\begin{eqnarray*}
30k+14-30k-13+i(12k+7-12k-6)&\equiv &i+1\; (\mbox{mod } 48k+22),\\
24k+12-12k-5+i(12k+7-12k-6)&\equiv &12k+7+i\; (\mbox{mod } 48k+22).
\end{eqnarray*}

\noindent These differences are all non-zero and distinct. In addition,
any two distinct differences $x$ and $y$ satisfy $x+y\not\equiv 0$ (mod $48k+22$),
By Lemmas \ref{distinct.entries.L2.2}, \ref{distinct.entries.L3.2} and the above arguments the
Latin squares ${\cal L}_1$, ${\cal L}_2$ and ${\cal L}_3$ are cyclic
MNOLSs.
\end{proof}

\section{Three cyclic  MNOLSs of Order $48k+38$, $k\geq 0$}
In this section we construct two cyclic Latin squares ${\cal L}_2$ and
${\cal L}_3$ both of order $48k+38$ and show that ${\cal L}_1$
(constructed by Example \ref{example1} with $n=48k+38$),
${\cal L}_2$ and ${\cal L}_3$ are cyclic MNOLSs.
The following lemma is crucial in this section.

\begin{lemma}\label{lemma:gcd.3}
Let $k$ be an integer. Working modulo $48k+38$,
{\bf 1.} $\gcd(6k+5,24k+19)=1$;
{\bf 2.} $\gcd(12k+11,48k+38)=1$;
{\bf 3.} $\gcd(6k+4,24k+19)=1$;
{\bf 4.} $\gcd(12k+9,48k+38)=1$.
\end{lemma}

\begin{proof}
The following equations verify the statements given in the lemma: $
{\bf 1.\ } 4(6k+5)-(24k+19) =1;\
{\bf 2.\ } (8k+7)(12k+11)-(2k+2)(48k+38) = 1;\
{\bf 3.\ } (8k+5)(6k+4)-(2k+1)(24k+19) = 1;\
{\bf 4.\ } (24k+17)(12k+9)-(6k+4)(48k+38) = 1.$
\end{proof}

\noindent Working modulo $48k+38$ we define the $(24k+19)\times 1$ matrices (column vectors)
$V_\alpha=[v_\alpha(i,0)]$, $\alpha=2,3$, by
\begin{eqnarray}
V_2&=&\{(2i,0,30k+23+i(12k+10))\mid 0\leq i\leq 12k+9\}\cup\nonumber\\
&&\{(2i+1,0,12k+9+i(12k+10))\mid 0\leq i\leq 12k+8\},\label{c2.3}\\
V_3&=&\{(2i,0,30k+24+i(12k+11))\mid 0\leq i\leq 12k+9\}\cup\nonumber\\
&&\{(2i+1,0,24k+20+i(12k+11))\mid 0\leq i\leq 12k+8\}.\label{c3.3}
\end{eqnarray}

\noindent For $\alpha=2,3$, let
${\cal C}_\alpha=V_\alpha\cup \overline{V_\alpha}$, where
$$\overline{V_\alpha}=\{(48k+37-i,0,48k+37-v_\alpha(i,0))\mid 0\leq
i\leq 24k+18\}.$$
Note that ${\cal C}_\alpha$ has the reflection property.
Now define ${\cal L}_\alpha=[l_\alpha(i,j)]$, where
$l_\alpha(i,j)\equiv {\cal C}_\alpha(i,0)+j$ (mod $48k+38$) for $0\leq i,j\leq 48k+37$.

\begin{lemma}\label{distinct.entries.L2.3}
The array ${\cal L}_2$ is a Latin square of order $48k+38$, $k\geq 0$.
\end{lemma}

\begin{proof}
 The entries in $V_2$ are all distinct   as verified by Equation \ref{eq:38-1}  for the case
 rows $2i$ and $2j$, where $0\leq i,j\leq 12k+9$, Equation \ref{eq:38-2}  for the case
rows $2i+1$ and $2j+1$,
where $0\leq i,j\leq 12k+8$,  and  Equation \ref{eq:38-3}  for the case
rows $2i$ and $2j+1$, where $0\leq i\leq 12k+9$ and $0\leq j\leq 12k+8$.
\begin{eqnarray}
30k+23+i(12k+10)&\equiv& 30k+23+j(12k+10)\;(\mbox{mod } 48k+38).\nonumber\\
\Rightarrow (j-i)(6k+5)&\equiv& 0 \;(\mbox{mod } 24k+19),\ \contradiction;\label{eq:38-1}\\
12k+9+i(12k+10)&\equiv &12k+9+j(12k+10)\;(\mbox{mod } 48k+38).\nonumber\\
\Rightarrow (j-i)(6k+5)&\equiv& 0 \;(\mbox{mod } 24k+19),\ \contradiction;\label{eq:38-2}\\
30k+23+i(12k+10)&\equiv& 12k+9+j(12k+10)\;(\mbox{mod } 48k+38),\nonumber\\
\Rightarrow
j-i&\equiv& 4(9k+7)\;(\mbox{mod } 24k+19),\label{eq:38-3}
\end{eqnarray}
implying $j=12k+9+i>12k+8$, which is   a contradiction.

 For any two rows containing entries $x$ and $y$ in $V_2$, $x+y+1\not\equiv 0$ (mod $48k+38$), specifically Equation  \ref{eq:38-4} for rows $2i$ and $2j$ of $V_2$, Equation \ref{eq:38-5} for rows  $2i+1$ and $2j+1$ and Equation \ref{eq:38-6} for rows $2i$ and $2j+1$.
 \begin{eqnarray}
 2(30k+23)+(i+j)(12k+10)+1&\equiv& 0\;(\mbox{mod } 48k+38), \nonumber\\
 \Rightarrow (i+j)(12k+10)&\equiv& -12k-9 \;(\mbox{mod } 48k+38),\ \contradiction;  \label{eq:38-4}\\
2( 12k+9)+(i+j)(12k+10)+1&\equiv& 0\;(\mbox{mod } 48k+38), \nonumber\\
 \Rightarrow (i+j)(12k+10)&\equiv& -24k-19 \;(\mbox{mod } 48k+38),\ \contradiction;\label{eq:38-5}\\
42k+32+(i+j)(12k+10)+1&\equiv& 0\;(\mbox{mod } 48k+38),  \nonumber\\
 \Rightarrow
 (i+j)(12k+10)&\equiv& -42k-33 \;(\mbox{mod } 48k+38),\ \contradiction.\label{eq:38-6}
 \end{eqnarray}

Thus the entries of ${\cal C}_2$ are all distinct and so
${\cal L}_2$ is a Latin square of order $48k+38$.
\end{proof}

\begin{lemma}\label{distinct.entries.L3.3}
The array ${\cal L}_3$ is a Latin square of order $48k+38$, $k\geq 0$.
\end{lemma}

\begin{proof}
 The entries in $V_3$ are all distinct  as verified by Equation \ref{eq:38-7}  for the case
 rows $2i$ and $2j$, where $0\leq i,j\leq 12k+9$, Equation \ref{eq:38-8}  for the case
rows $2i+1$ and $2j+1$,
where $0\leq i,j\leq 12k+8$, and  Equation \ref{eq:38-9}  for the case
rows $2i$ and $2j+1$, where $0\leq i\leq 12k+9$ and $0\leq j\leq 12k+8$.
\begin{eqnarray}
30k+24+i(12k+11)&\equiv& 30k+24+j(12k+11)\;(\mbox{mod } 48k+38),\nonumber\\
 \Rightarrow
 (j-i)(12k+11)&\equiv& 0 \;(\mbox{mod } 48k+38),\ \contradiction;\label{eq:38-7}\\
 24k+20+i(12k+11)&\equiv &24k+20+j(12k+11)\;(\mbox{mod } 48k+38),\nonumber\\
 \Rightarrow (j-i)(12k+11)&\equiv& 0 \;(\mbox{mod } 48k+38),\ \contradiction;\label{eq:38-8}\\
 30k+24+i(12k+11)&\equiv& 24k+20+j(12k+11)\;(\mbox{mod } 48k+38),\nonumber\\
 \Rightarrow
 j-i\equiv (8k+7)(6k+4)&\equiv &36k+28\;(\mbox{mod } 48k+38),\label{eq:38-9}
\end{eqnarray}
implying  $j=36k+28+i>12k+8$ or $j=-12k-10+i$, which is a contradiction.

 For any two rows containing entries $x$ and $y$ in $V_3$, $x+y+1\not\equiv 0$ (mod $48k+38$), specifically Equation  \ref{eq:38-10} for rows $2i$ and $2j$ of $V_3$, Equation \ref{eq:38-11} for rows  $2i+1$ and $2j+1$ and Equation \ref{eq:38-12} for rows $2i$ and $2j+1$.
 \begin{eqnarray}
2( 30k+24)+(i+j)(12k+11)+1&\equiv& 0\;(\mbox{mod } 48k+38),\nonumber\\
 \Rightarrow  (i+j+1)(12k+11) &\equiv& 0 \;(\mbox{mod } 48k+38),\ \contradiction\label{eq:38-10}\\
2( 24k+20)+(i+j)(12k+11)+1&\equiv& 0\;(\mbox{mod } 48k+38),\nonumber\\
 \Rightarrow i+j\equiv (-3)(8k+7)&\equiv &24k+17 \;(\mbox{mod } 48k+38),\ \contradiction;\label{eq:38-11}\\
54k+44+(i+j)(12k+11)+1&\equiv& 0\;(\mbox{mod } 48k+38),\nonumber\\
 \Rightarrow i+j\equiv(-6k-7)(8k+7)&\equiv& 36k+27\;(\mbox{mod } 48k+38),\ \contradiction.\label{eq:38-12}
 \end{eqnarray}

Thus the entries of ${\cal C}_3$ are all distinct and so
${\cal L}_3$ is a Latin square of order $48k+38$.
\end{proof}

\begin{theorem}\label{theorem.3}
The Latin squares ${\cal L}_1$, ${\cal L}_2$ and ${\cal L}_3$ are cyclic
MNOLSs of order $48k+38$, $k\geq 0$.
\end{theorem}

\begin{proof}
Respectively, for rows $2i$ and $2i+1$ the differences between entries  of $V_2$ and $V_1$, are
\begin{eqnarray*}
30k+23+i(12k+10)-2i&\equiv& 30k+23+i(12k+8)\; (\mbox{mod } 48k+38),\\
12k+9+i(12k+10)-2i-1&\equiv &12k+8+i(12k+8)\; (\mbox{mod } 48k+38).
\end{eqnarray*}

\noindent These differences are all non-zero because in the first instance
$30k+23$ is odd but $12k+8$ and $48k+38$ are even and in the second
if $(i+1)(12k+8)\equiv 0$
(mod $48k+38$), then by Lemma \ref{lemma:gcd.3}, $i+1\equiv 0$ (mod $24k+19$),
which implies $i=24k+18$, a contradiction.

\noindent The differences are all distinct  as verified by Equation \ref{eq:38-13} for
 rows $2i$ and $2j$ and for rows $2i+1$ and $2j+1$, and using parity conditions in Equation \ref{eq:38-14}  for rows $2i$ and $2j+1$.
\begin{eqnarray}
(j-i)(6k+4)&\equiv& 0\;(\mbox{mod } 24k+19),\ \contradiction;\label{eq:38-13} \\
(j-i)(12k+8)&\equiv& 18k+15 \;(\mbox{mod } 48k+38),\ \contradiction.\label{eq:38-14}
 \end{eqnarray}

\noindent In addition,
any two distinct differences $x$ and $y$, produced by corresponding rows in $V_2$ and $V_1$,
satisfy $x+y\not\equiv 0$ (mod $48k+38$) as verified by Equation \ref{eq:38-15} for rows
$2i$ and $2j$, Equation \ref{eq:38-16} for $2i+1$ and $2j+1$ and using parity conditions together with Equation \ref{eq:38-17} for rows $2i$ and $2j+1$. In all such cases $x+y$ is congruent to
\begin{eqnarray}
(i+j+1)(6k+4)&\equiv& 0\;(\mbox{mod } 24k+19),\ \contradiction;\label{eq:38-15} \\
 (i+j+2)(6k+4)&\equiv& 0\;(\mbox{mod } 24k+19),\ \contradiction;\label{eq:38-16} \\
 (42k+31)+(i+j)(12k+8)&\equiv& 0\;(\mbox{mod } 24k+19),\ \contradiction.\label{eq:38-17}
\end{eqnarray}

Respectively, for rows $2i$ and $2i+1$ the differences between entries  of $V_3$ and $V_1$, are
\begin{eqnarray*}
30k+24+i(12k+11)-2i&\equiv& 30k+24+i(12k+9)\; (\mbox{mod } 48k+38),\\
24k+20+i(12k+11)-2i-1&\equiv& 24k+19+i(12k+9)\; (\mbox{mod } 48k+38).
\end{eqnarray*}

\noindent These differences are all non-zero because in the first instance
if $30k+24+i(12k+9)\equiv 0$ (mod $48k+38$), then by the proof of Lemma
\ref{lemma:gcd.3}, $i\equiv (-30k-24)(24k+17)\equiv 12k+10$ (mod $48k+38$), a contradiction and in the second instance
if $24k+19+i(12k+9)\equiv 0$ (mod $48k+38$), then by the proof of Lemma \ref{lemma:gcd.3},
$i\equiv (-24k-19)(24k+17)\equiv 24k+19$ (mod $48k+38$), a contradiction.

\noindent The differences are all distinct  as verified by Equation \ref{eq:38-18} for
 rows $2i$ and $2j$ and for rows $2i+1$ and $2j+1$, and  Equation \ref{eq:38-19}  for rows $2i$ and $2j+1$.
\begin{eqnarray}
(j-i)(12k+9)&\equiv& 0\; (\mbox{mod } 48k+38),\ \contradiction;\label{eq:38-18} \\
j-i\equiv (6k+5)(24k+17)&\equiv& 12k+9 \; (\mbox{mod } 48k+38),\label{eq:38-19}
\end{eqnarray}
implying  $j=12k+9+i>12k+8$, which is  a contradiction.

\noindent Finally,
for any two distinct differences $x$ and $y$ produced by corresponding rows in $V_3$ and $V_1$, $x+y\not\equiv 0$ (mod $48k+38$) as verified by Equation \ref{eq:38-20} for rows
$2i$ and $2j$, Equation \ref{eq:38-21} for $2i+1$ and $2j+1$ and parity conditions in Equation \ref{eq:38-22} for rows $2i$ and $2j+1$. In all such cases $x+y$ is congruent to
\begin{eqnarray}
12k+10+(i+j)(12k+9)&\equiv& 0\; (\mbox{mod } 48k+38),\nonumber\\
 \Rightarrow i+j\equiv (-12k-10)(24k+17)&\equiv& 24k+20\; (\mbox{mod } 48k+38),  \ \contradiction;\label{eq:38-20} \\
 (i+j)(12k+9) &\equiv& 0\; (\mbox{mod } 48k+38),\ \contradiction;\label{eq:38-21} \\
 6k+5+(i+j)(12k+9)&\equiv& 0\; (\mbox{mod } 48k+38),\nonumber\\
 \Rightarrow i+j\equiv (-6k-5)(24k+17)&\equiv& 36k+29  \; (\mbox{mod } 48k+38),\ \contradiction.\label{eq:38-22}
\end{eqnarray}

Respectively, the differences between entries are in rows $2i$ and $2i+1$ of $V_3$ and $V_2$, are
\begin{eqnarray*}
30k+24-30k-23+i(12k+11-12k-10)&\equiv& i+1\; (\mbox{mod } 48k+38)\\
24k+20-12k-9+i(12k+11-12k-10)&\equiv& 12k+11+i\; (\mbox{mod } 48k+38).
\end{eqnarray*}

\noindent These differences  are all non-zero and distinct. In addition,
any two distinct differences $x$ and $y$ satisfy $x+y\not\equiv 0$ (mod $48k+38$).

By Lemmas \ref{distinct.entries.L2.3}, \ref{distinct.entries.L3.3} and the above arguments, the
Latin squares ${\cal L}_1$, ${\cal L}_2$ and ${\cal L}_3$ are cyclic
MNOLSs.
\end{proof}


\section{Three cyclic MNOLSs of Order $48k+46$, $k\geq 0$}

In this section we construct two cyclic Latin squares ${\cal L}_2$ and
${\cal L}_3$ both of order $48k+46$ and show that ${\cal L}_1$
(constructed by Example \ref{example1} with $n=48k+46$),
${\cal L}_2$ and ${\cal L}_3$ are cyclic MNOLSs.
The following lemma is crucial in this section.

\begin{lemma}\label{lemma:gcd.4}
Let $k$ be an integer. Working modulo $48k+46$,
{\bf 1.} $\gcd(6k+6,24k+23)=1$;
{\bf 2.} $\gcd(12k+13,48k+46)=1$;
{\bf 3.}$\gcd(6k+5,24k+23)=1$;
{\bf 4.} $\gcd(12k+11,48k+46)=1$.
\end{lemma}

\begin{proof}
The following equations verify the statements given in the lemma: $
{\bf 1.\ } 4(6k+6)-(24k+23)=1;\
{\bf 2.\ } (-8k-7)(12k+13)+(2k+2)(48k+46)=1;\
{\bf 3.\ } (-8k-9)(6k+5)+(2k+2)(24k+23)=1;\
{\bf 4.\ } (24k+21)(12k+11)+(6k+6)(48k+46)=1.$
\end{proof}

\noindent Working modulo $48k+46$ we define the $(24k+23)\times 1$ matrices (column vectors)
$V_\alpha=[v_\alpha(i,0)]$, $\alpha=2,3$, by
\begin{eqnarray}
V_2&=&\{(2i,0,6k+5+i(12k+12))\mid 0\leq i\leq 12k+11\}\cup\nonumber\\
&&\{(2i+1,0,12k+11+i(12k+12))\mid 0\leq i\leq 12k+10\},\label{c2.4}\\
V_3&=&\{(2i,0,6k+6+i(12k+13))\mid 0\leq i\leq 12k+11\}\cup\nonumber\\
&&\{(2i+1,0,24k+24+i(12k+13))\mid 0\leq i\leq 12k+10\}.\label{c3.4}
\end{eqnarray}
 For $\alpha=2,3$, let
${\cal C}_\alpha=V_\alpha\cup \overline{V_\alpha}$, where
$$\overline{V_\alpha}=\{(48k+45-i,0,48k+45-c_\alpha(i,0))\mid 0\leq
i\leq 24k+22\}.$$
\noindent Note that ${\cal C}_\alpha$ has the reflection property.
Now define ${\cal L}_\alpha=[l_\alpha(i,j)]$, where
$l_\alpha(i,j)\equiv {\cal C}_\alpha(i,0)+j$ (mod $48k+46$) for $0\leq i,j\leq 48k+45$.

\begin{lemma}\label{distinct.entries.L2.4}
The array ${\cal L}_2$ is a Latin square of order $48k+46$, $k\geq 0$.
\end{lemma}

\begin{proof}
 The entries in $V_2$ are all distinct
 as verified by Equation \ref{eq:46-1}  for the case
 rows $2i$ and $2j$, where $0\leq i,j\leq 12k+11$, Equation \ref{eq:46-2}  for
rows $2i+1$ and $2j+1$,
where $0\leq i,j\leq 12k+10$, and  Equation \ref{eq:46-3}  for the case
rows $2i$ and $2j+1$, where $0\leq i\leq 12k+11$ and $0\leq j\leq 12k+10$.
\begin{eqnarray}
6k+5+i(12k+12)&\equiv& 6k+5+j(12k+12)\;(\mbox{mod } 48k+46),\nonumber\\
\Rightarrow (j-i)(6k+6)&\equiv&0 \;(\mbox{mod } 24k+23),\ \contradiction;\label{eq:46-1}\\
12k+11+i(12k+12)&\equiv& 12k+11+j(12k+12)\;(\mbox{mod } 48k+46),\nonumber\\
\Rightarrow (j-i)(6k+6)&\equiv&0 \;(\mbox{mod } 24k+23),\ \contradiction;\label{eq:46-2}\\
6k+5+i(12k+12)&\equiv &12k+11+j(12k+12)\;(\mbox{mod } 48k+46),\nonumber\\
\Rightarrow j-i&\equiv& 4(-3k-3)\;(\mbox{mod } 24k+23),\ \contradiction.\label{eq:46-3}
\end{eqnarray}
implying  $j=12k+11+i>12k+10$, which is a contradiction.

For any two rows containing entries $x$ and $y$ in $V_2$, parity conditions  and the following equations can be used to verify  $x+y+1\not\equiv 0$ (mod $48k+46$), specifically
 Equation \ref{eq:46-4}
for rows $2i$ and $2j$, Equation \ref{eq:46-5} for rows $2i+1$ and $2j+1$ and Equation \ref{eq:46-6} for rows $2i$ and $2j+1$.
\begin{eqnarray}
2(6k+5)+(i+j)(12k+12)+1&\equiv& 0\;(\mbox{mod } 48k+46),\nonumber\\
\Rightarrow
(i+j)(12k+12)&\equiv&-12k-11 \;(\mbox{mod } 48k+46),\ \contradiction;\label{eq:46-4}\\
2(12k+11)+(i+j)(12k+12)+1&\equiv& 0\;(\mbox{mod } 48k+46),\nonumber\\
\Rightarrow (i+j)(12k+12)&\equiv&-24k-23 \;(\mbox{mod } 48k+46),\ \contradiction;\label{eq:46-5}\\
18k+16+(i+j)(12k+12)+1&\equiv &0\;(\mbox{mod } 48k+46),\nonumber\\
\Rightarrow (i+j)(12k+12)&\equiv&-18k-17 \;(\mbox{mod } 48k+46),\ \contradiction.\label{eq:46-6}
\end{eqnarray}
Thus the entries of ${\cal C}_2$ are all distinct and so
${\cal L}_2$ is a Latin square of order $48k+46$.
\end{proof}

\begin{lemma}\label{distinct.entries.L3.4}
The array ${\cal L}_3$ is a Latin square of order $48k+46$, $k\geq 0$.
\end{lemma}

\begin{proof}
The entries in $V_3$ are all distinct  as verified by Equation \ref{eq:46-7}  for the case
 rows $2i$ and $2j$, where $0\leq i,j\leq 12k+11$, Equation \ref{eq:46-8}  for the case
rows $2i+1$ and $2j+1$,
where $0\leq i,j\leq 12k+10$, and  Equation \ref{eq:46-9}  for the case
rows $2i$ and $2j+1$, where $0\leq i\leq 12k+11$ and $0\leq j\leq 12k+10$.
\begin{eqnarray}
6k+6+i(12k+13)&\equiv& 6k+6+j(12k+13)\;(\mbox{mod } 48k+46),\nonumber\\
\Rightarrow (j-i)(12k+13)&\equiv&0 \;(\mbox{mod } 48k+46),\ \contradiction;\label{eq:46-7}\\
24k+24+i(12k+13)&\equiv& 24k+24+j(12k+13)\;(\mbox{mod } 48k+46),\nonumber\\
\Rightarrow (j-i)(12k+13)&\equiv&0 \;(\mbox{mod } 48k+46),\ \contradiction;\label{eq:46-8}\\
6k+6+i(12k+13)&\equiv& 24k+24+j(12k+13)\;(\mbox{mod } 48k+46),\nonumber\\
\Rightarrow j-i\equiv (-8k-7)(-18k-18)&\equiv& 36k+34\;(\mbox{mod } 48k+46),\label{eq:46-9}
\end{eqnarray}
implying $j=36k+34+i>12k+10$ or $j=-12k-12+i<0$,  which is a contradiction.

For any two rows containing entries $x$ and $y$ in $V_3$, $x+y+1\not\equiv 0$ (mod $48k+46$) as verified by Equation \ref{eq:46-10}
for rows $2i$ and $2j$, Equation \ref{eq:46-11} for rows $2i+1$ and $2j+1$ and Equation \ref{eq:46-12} for rows $2i$ and $2j+1$.
\begin{eqnarray}
2(6k+6)+(i+j)(12k+13)+1&\equiv& 0\;(\mbox{mod } 48k+46),\nonumber\\
\Rightarrow (i+j+1)(12k+13)&\equiv&0 \;(\mbox{mod } 48k+46),\ \contradiction;\label{eq:46-10}\\
2(24k+24)+(i+j)(12k+13)+1&\equiv& 0\;(\mbox{mod } 48k+46),\nonumber\\
\Rightarrow i+j\equiv (-3)(-8k-7)&\equiv& 24k+21\;(\mbox{mod } 48k+46),\ \contradiction;\label{eq:46-11}\\
30k+30+(i+j)(12k+13)+1&\equiv& 0\;(\mbox{mod } 48k+46),\nonumber\\
\Rightarrow i+j\equiv (-30k-31)(-8k-7)&\equiv& 36k+33\;(\mbox{mod } 48k+46),\ \contradiction.\label{eq:46-12}
\end{eqnarray}
Thus the entries of ${\cal C}_3$ are all distinct and so
${\cal L}_3$ is a Latin square of order $48k+46$.
\end{proof}

\begin{theorem}\label{theorem.4}
The Latin squares ${\cal L}_1$, ${\cal L}_2$ and ${\cal L}_3$ are
cyclic MNOLSs of order $48k+46$, $k\geq 0$.
\end{theorem}

\begin{proof}
Respectively, the differences between entries in rows $2i$ and $2i+1$ of $V_2$ and $V_1$, are
\begin{eqnarray*}
6k+5+i(12k+12)-2i&\equiv& 6k+5+i(12k+10)\; (\mbox{mod } 48k+46),\\
12k+11+i(12k+12)-2i-1&\equiv& 12k+10+i(12k+10)\; (\mbox{mod } 48k+46).
\end{eqnarray*}

\noindent These differences are all non-zero since in the first instance
2 divides $12k+10$ and $48k+46$ but does not divides $6k+5$ and in the second if $(i+1)(12k+10)\equiv 0$
(mod $48k+46$), then by Lemma \ref{lemma:gcd.4}, $i+1\equiv 0$ (mod $24k+23$),
which implies $i=24k+12$, a contradiction.

\noindent The differences are all distinct  as verified by Equation \ref{eq:46-13} for
 rows $2i$ and $2j$ and for rows $2i+1$ and $2j+1$, and  parity conditions in Equation \ref{eq:46-14}  for rows $2i$ and $2j+1$.
 \begin{eqnarray}
 (j-i)(12k+10)&\equiv& 0\;(\mbox{mod } 48k+46),\nonumber\\
 \Rightarrow (j-i)(6k+5)&\equiv& 0\;(\mbox{mod } 24k+23),\ \contradiction;\label{eq:46-13}\\
(j-i)(12k+10)&\equiv& -6k-5\;(\mbox{mod } 48k+46),\ \contradiction.\label{eq:46-14}
 \end{eqnarray}

\noindent In addition, any two different rows the two differences $x$ and
$y$, produced by corresponding rows of $V_2$ and $V_1$, satisfy
 $x+y\not\equiv 0$ (mod $48k+46$), as verified by Equation \ref{eq:46-15} for rows
$2i$ and $2j$, Equation \ref{eq:46-16} for rows $2i+1$ and $2j+1$ and parity
conditions in Equation \ref{eq:46-17} for rows $2i$ and $2j+1$. In all such cases $x+y$ is congruent to
\begin{eqnarray}
(i+j+1)(6k+5)&\equiv& 0\;(\mbox{mod } 24k+23),\ \contradiction;\label{eq:46-15}\\
(i+j+2)(6k+5)&\equiv& 0\;(\mbox{mod } 24k+23),\ \contradiction;\label{eq:46-16}\\
(18k+15)+(i+j)(12k+10)&\equiv& 0\;(\mbox{mod } 48k+46),\ \contradiction.\label{eq:46-17}
\end{eqnarray}

Respectively, the differences between entries in rows $2i$ and $2i+1$ of $V_3$ and $V_1$, are
\begin{eqnarray*}
6k+6+i(12k+13)-2i&\equiv& 6k+6+i(12k+11)\; (\mbox{mod } 48k+46),\\
24k+24+i(12k+13)-2i-1&\equiv& 24k+23+i(12k+11)\; (\mbox{mod } 48k+46).
\end{eqnarray*}

\noindent Equations  \ref{eq:46-18} and  \ref{eq:46-19} verify that these differences are all non-zero.
\begin{eqnarray}
i\equiv (-6k-6)(24k+21)&\equiv& 12k+12\;(\mbox{mod } 48k+46),\ \contradiction;\label{eq:46-18}\\
i\equiv (-24k-23)(24k+21)&\equiv& 24k+23\;(\mbox{mod } 48k+46),\ \contradiction.\label{eq:46-19}
\end{eqnarray}

\noindent If two differences produced by rows $2i$ and $2j$ or by rows $2i+1$ and $2j+1$ are equal, then
$(j-i)(12k+11)\equiv 0$ (mod $48k+46$).
Now by Lemma \ref{lemma:gcd.4}, $i=j$.
Equation \ref{eq:46-20} verifies that   two differences produced by rows $2i$ and $2j+1$ are never equal.
\begin{eqnarray}
(j-i)(12k+11)&\equiv& -18k-17\;(\mbox{mod } 48k+46),\nonumber\\
 \Rightarrow  j-i\equiv (-18k-17)(24k+21)&\equiv& 12k+11 \;(\mbox{mod } 48k+46),\label{eq:46-20}
\end{eqnarray}
implying $j=12k+11+i>12k+10$, which is a contradiction.

\noindent In addition, any two different rows the two differences $x$ and
$y$, produced by corresponding rows of $V_3$ and $V_1$, satisfy
 $x+y\not\equiv 0$ (mod $48k+46$), as verified by Equation \ref{eq:46-21} for rows
$2i$ and $2j$, Equation \ref{eq:46-22} for rows $2i+1$ and $2j+1$
and parity arguments together with Equation \ref{eq:46-23} for rows $2i$ and $2j+1$. In all such cases $x+y$ is congruent to
\begin{eqnarray}
12k+12+(i+j)(12k+11) &\equiv& 0\;(\mbox{mod } 48k+46),\nonumber\\
 \Rightarrow i+j\equiv (-12k-12)(24k+21)&\equiv& 24k+24\;(\mbox{mod } 48k+46),\ \contradiction;\label{eq:46-21}\\
(i+j)(12k+11)&\equiv& 0\;(\mbox{mod } 48k+46),\ \contradiction;\label{eq:46-22}\\
30k+29+(i+j)(12k+11)&\equiv& 0\;(\mbox{mod } 48k+46),\nonumber\\
 \Rightarrow i+j\equiv (-30k-29)(24k+21)&\equiv &36k+35\;(\mbox{mod } 48k+46),\ \contradiction.\label{eq:46-23}
\end{eqnarray}

Respectively, the differences between entries are in rows $2i$ and $2i+1$ of $V_3$ and $V_2$, are
\begin{eqnarray*}
6k+6-6k-5+i(12k+13-12k-12)&\equiv& i+1\; (\mbox{mod } 48k+46),\\
24k+24-12k-11+i(12k+13-12k-12)&\equiv &12k+13+i\; (\mbox{mod } 48k+46).
\end{eqnarray*}
These are all non-zero and distinct. In addition,
any two distinct differences $x$ and $y$ satisfy $x+y\not\equiv 0$ (mod $48k+46$).

By Lemmas \ref{distinct.entries.L2.4}, \ref{distinct.entries.L3.4} and the above arguments, the
Latin squares ${\cal L}_1$, ${\cal L}_2$ and ${\cal L}_3$ are
cyclic MNOLSs.
\end{proof}


\begin{thebibliography}{9}
\bibitem{colbourn} C.J. Colbourn and J.H. Dinitz,  (Eds.), Handbook of combinatorial designs. CRC
press, 2010.


\bibitem{evans} A.B. Evans, Orthomorphism graphs of groups, Lecture Notes in Mathematics, volume 1535, Springer-Verlag, 1992.

 \bibitem {LvR}  P.C. Li and G.H.J. van Rees, {\em Nearly Orthogonal Latin Squares}
 Journal of Combinatorial Mathematics and Combinatorial Computing {\bf 62} (2007), 13--24.


\bibitem{RSS} D. Raghavarao, S.S. Shrikhande and M.S. Shirkhande,
{\em Incidence Matrices and Inequalities for Combinatorial Designs}
 Journal of Combinatorial Design  {\bf 10} (2002), 17--26.


 \bibitem{todorov}  D.T. Todorov, {\em Four Mutually  Orthogonal Latin Squares of Order 14}
 Journal of Combinatorial Design {\bf 20} (2012), 363-367.

 \bibitem{wanless} I.M. Wanless, {\em Atomic Latin squares based on cyclotomic orthomorphisms},
 Electronic Journal of Combinatorics {\bf 12} (2005), R22.
\end{thebibliography}
\end{document}